\documentclass[12pt]{amsart}
\usepackage[utf8]{inputenc}
\usepackage[foot]{amsaddr}
\usepackage{amsmath,amsfonts,amsthm,amssymb,verbatim,etoolbox,color}
\usepackage{enumitem} 
\usepackage{epigraph}
\usepackage{mathtools}
\usepackage{appendix}
\usepackage{xfrac}
\usepackage{flexisym}
\usepackage{bbm}
\usepackage{comment}
\usepackage{tikz}
\usetikzlibrary{patterns}

\usepackage[margin=1in]{geometry}
\usepackage[
    backend=biber,
    maxnames=8,
    maxalphanames=8,
    style=alphabetic,
    backref
  ]{biblatex}
\usepackage{scalerel}

\usepackage{hyperref}
\hypersetup{
    colorlinks=true,
    linkcolor=red,
    filecolor=magenta,      
    urlcolor=magenta,
    pdftitle={},
    pdfauthor = {Pierre-Alexandre Bazin, Ihor Pylaiev, Fred Tyrrell},
    pdfpagemode=FullScreen,
    linktocpage = true,
    citecolor = blue,
    }

\usepackage{theoremref}

\newtheorem{thm}{Theorem}
\newtheorem{lemma}[thm]{Lemma}
\newtheorem{lem}[thm]{Lemma}

\newtheorem{corollary}[thm]{Corollary}
\newtheorem{conjecture}[thm]{Conjecture}
\newtheorem{prop}[thm]{Proposition}
\theoremstyle{definition}
\newtheorem{defn}[thm]{Definition}
\newtheorem{remark}[thm]{Remark}

\numberwithin{thm}{section}

\newcommand{\N}{\mathbb{N}}
\newcommand{\D}{\mathbb{D}}
\newcommand{\U}{\mathbb{U}}

\newcommand{\Z}{\mathbb{Z}}
\newcommand{\R}{\mathbb{R}}
\renewcommand{\P}{\mathbb{P}}

\newcommand{\Q}{\mathbb{Q}}
\newcommand{\C}{\mathbb{C}}
\newcommand{\T}{\mathbb{T}}

\newcommand{\PP}{\mathcal{P}}

\newcommand{\sub}{\subseteq}

\newcommand{\bra}[1]{\left(#1\right)}

\newcommand{\red}[1]{\color{red}#1\color{black}}

\DeclarePairedDelimiter\floor{\lfloor}{\rfloor}

\newcommand{\ra}{\rightarrow}

\newcommand{\conj}[1]{\overline{#1}}
\newcommand{\abs}[1]{\left\lvert {#1} \right \rvert}
\newcommand{\e}{\mathrm{e}}
\newcommand{\Ind}{\mathbbm{1}}

\newcommand{\al}{\alpha}

\newcommand{\ep}{\varepsilon}
\newcommand{\be}{\beta}
\newcommand{\de}{\delta}

\newcommand{\ph}{\varphi}

\title{Bounded Exponential Sums with Multiplicative Coefficients}
\author{Pierre-Alexandre Bazin}
\email{bazin@imj-prg.fr}
\address{Institut de Math\'ematiques de Jussieu-Paris Rive Gauche, Universit\'e Paris Cit\'e, Sorbonne Universit\'e, CNRS}
\author{Ihor Pylaiev}
\email{ip406@cam.ac.uk}
\address{Faculty of Mathematics, University of Cambridge}
\author{Fred Tyrrell}
\email{fred.tyrrell@bristol.ac.uk}
\address{Fry Building, School of Mathematics, University of Bristol}
\date{\today}

\begin{document}

\begin{abstract}
We investigate when the exponential sum $S_f(x,\alpha) := \sum_{n\le x}f(n)\e(n\alpha)$ is bounded, for a multiplicative function $f$ and $\alpha\in\mathbb{R}$. We show that under natural assumptions, $S_f(x,\alpha)$ is bounded only when $f$ is very close to a twisted Dirichlet character $\chi(n)n^{it}$. We obtain sharper classification results for functions that are completely multiplicative or take only finitely many values, including a complete classification in the case when $f$ is completely multiplicative and $\al$ is irrational. We also prove a stronger classification under the assumption that the sum is bounded for a positive measure set of $\alpha$.
\end{abstract}

\maketitle
\tableofcontents

\section{Introduction}
\subsection{Background and motivation}
Exponential sums involving multiplicative functions\footnote{A function $f: \N \ra \C$ is \emph{multiplicative} if $f(mn) = f(m)f(n)$ for all coprime $m,n \in \N$. If $f(mn)=f(m)f(n)$ for all $m,n \in \N$, we say $f$ is \emph{completely multiplicative.}} form a central object of study in analytic number theory, bridging additive and multiplicative structures in the integers. For a multiplicative function \( f : \mathbb{N} \to \mathbb{C} \), we consider sums of the form
\[
S_f(x, \alpha) := \sum_{n \leq x} f(n) \e(n\alpha),
\]
where \( \e(\alpha) = \e^{2\pi i \alpha} \) and \( \alpha \in \mathbb{R} \) is fixed. Understanding the size and behaviour of exponential sums with multiplicative coefficients sheds light on how multiplicative functions interact with additive characters, with applications to questions of equidistribution modulo 1, the pseudorandomness of arithmetic functions, and the classification of multiplicative functions according to their additive correlations.
\bigbreak
A classical result of Daboussi \cite{daboussi1975fonctions} shows that if \( f \) is multiplicative and bounded, then \( S_f(x, \alpha) = o(x) \) for all irrational \( \alpha \). Montgomery and Vaughan \cite{montgomery1977exponential} further developed general bounds for such sums, achieving a close to optimal upper bound for these sums on major arcs. Several generalisations and extensions of these results have since appeared, and exponential sums involving multiplicative functions remain an active area of research within analytic number theory; see, for example, the work of Daboussi-Delange \cite{daboussi1982multiplicative}, K{\'a}tai \cite{katai1986remark}, Bachmann \cite{bachman1999exponential}, and Matom{\"a}ki-Radziwi{\l}{\l}-Tao-Ter{\"a}v{\"a}inen-Ziegler \cite{matomaki2023higher}. Recent research has also been conducted on exponential sums with random multiplicative coefficients \cite{hardy2024bounds,benatar2022moments},  as well as the distribution of exponential character sums \cite{bober2025distribution}.
\bigbreak
More recently, Granville and Soundararajan \cite{granville2014multiplicative} formalised the orthogonality between additive characters and multiplicative functions via the \emph{pretentious approach}, introducing a quantitative framework to measure the extent to which a multiplicative function mimics a structured object such as \( n^{it} \chi(n) \). This perspective unifies and generalises many earlier results by showing that cancellation in exponential sums is intimately tied to the `distance' between \( f \) and such structured functions, and was recently used by La Bret{\`e}che and Granville \cite{de2022exponential} to continue the study of these exponential sums over minor arcs.
\bigbreak
This paper investigates the most extreme version of this cancellation phenomenon, namely when $S_f\bra{x,\al}$ is bounded. This is motivated by work on \emph{discrepancy} problems, in particular Tao's solution to the Erdős Discrepancy Problem \cite{tao2015erdos}. 
\smallbreak
The Erdős Discrepancy Problem was a long-standing conjecture of Erdős \cite{erdos}, which states that for \emph{any} \( f: \mathbb{N} \to \{-1, +1\} \), the discrepancy
\[
\sup_{d,n} \left| \sum_{j=1}^n f(dj) \right|
\]
is unbounded. One of the key ingredients of Tao's solution was the breakthrough of Matom{\"a}ki-Radziwi{\l}{\l} \cite{MR} and a result from the first Polymath project, which showed that solving the problem in the case where $f$ is completely multiplicative was sufficient to prove the full conjecture. Numerous investigations into the discrepancy of various arithmetic functions have since appeared in the literature, including \cite{klurman2017correlations}, \cite{klurman2018rigidity}, \cite{klurman2021multiplicative}, \cite{klurman2024beyond}, \cite{aymone2022complex}, \cite{Aymone} and \cite{frantzikinakis2019good}. For further background on Tao's solution to the Erd\H{o}s Discrepancy Problem, see the survey article of Soundararajan \cite{soundararajan2018tao}.
\smallbreak
Particularly relevant to the work in this paper is the work of Frantzikinakis \cite{frantzikinakis2019good} on the \emph{weighted} Erdős discrepancy problem, which extends Tao’s ideas by asking for which weight functions \( w: \mathbb{N} \to \mathbb{C} \) it is possible to find sequences \( f(n) \in \{-1,+1\} \) such that the weighted discrepancy
\[
\sup_{d,n} \left| \sum_{j=1}^n f(dj) w(j) \right|
\]
remains bounded. In particular, Frantzikinakis demonstrated that the weighted discrepancy is unbounded when the weight function is of the form \( w(n) = \e(n^k \alpha) \) for \( k \geq 2 \), using ergodic methods. The story when $k=1$, however, is rather different --- indeed simply letting $f=1$ and $\al \notin \Z$, we have 
\[\abs{S_1(x,\al)} = \abs{\sum\limits_{n \leq x}\e\bra{n \al}} = \abs{\e\bra{\al}\frac{1-\e\bra{\floor*{x}\al}}{1-\e\bra{\al}}} \leq \frac{2}{\abs{1-\e\bra{\al}}},\]
so the sum can be bounded in this case.

\subsection{Main results}
In this paper, we investigate when exponential sums of the form $S_f(x,\al)$ can remain bounded as \( x \to \infty \), for a fixed value of \( \alpha \). Our primary aim is to classify the multiplicative functions \( f \) for which
\[
\sup\limits_{x}\abs{S_f(x,\al)} < \infty.
\]
Our first main result provides a general classification of multiplicative functions with bounded exponential sums.\footnote{Here, and throughout, $\U = \{z \in \C: \abs{z}=1\}$ is the unit circle in $\C$. We will also use $\T = \R/\Z$ to denote the torus.}
\begin{thm}\thlabel{stronger}
    Let $f: \N \ra \U \cup \{0\}$ be a multiplicative function, and assume that
    \[\sum_{p:f(p) = 0}\frac{1}{p} < \infty.\]
   If, for some $\al \in \R$,
    \[\sup\limits_{x}\abs{S_f(x,\al)} < \infty,\]
    then there is a Dirichlet character $\chi$ and a real number $t \in \R$ such that
    \[\sum_{p} \max_{k}\abs{f(p^k)-\chi(p^k)p^{ikt}} < \infty.\]
\end{thm}

We have the following special cases of \thref{stronger}, when $f$ is a \emph{completely} multiplicative function, takes finitely many values, or both ---- in these cases, we can say something stronger than \thref{stronger}.
\begin{thm}\thlabel{special}
    Let $f: \N \ra \U \cup \{0\}$ be a multiplicative function, and assume that 
    \[\sum_{p:f(p)=0}\frac{1}{p} < \infty.\]If, for some $\al \in \R \smallsetminus \Q$,
    \[\sup\limits_{x}\abs{S_f(x,\al)} < \infty,\]
    then there is some Dirichlet character $\chi$ and a real number $t \in \R$ such that:
\begin{enumerate}[label=\roman*)]
    \item If $f$ is \emph{completely} multiplicative, $f(p) = \chi(p)p^{it}$, for all $p$. 

\item If $f$ takes only finitely many values, $f(p^k) = f(p^{k-1})\chi(p)$ for all but finitely many prime powers $p^k$.
        \item If $f$ is both completely multiplicative and takes finitely many values, $f(p) = \chi(p)$ for all $p$.
    \end{enumerate}

\end{thm}

Our final main result shows that, if we include the stronger assumption that the exponential sum is bounded for a positive proportion of $\al$, then we can say more than \thref{stronger}. We have the following result for general multiplicative functions.
\begin{thm}\thlabel{bazinmain}
    Let $f: \N \ra \U \cup \{0\}$ be a multiplicative function, and assume that
    \[\sum_{p:f(p) = 0}\frac{1}{p} < \infty.\]
   If, for a positive measure set of $\al \in \R$,
    \[\sup\limits_{x}\abs{S_f(x,\al)} < \infty,\]
    then there is a Dirichlet character $\chi$ and a real number $t \in \R$ such that
    \[\sum_{p} \sum_k \abs{f(p^k)- f(p^{k-1})\chi(p)p^{it}} < \infty.\]
\end{thm}

\subsection{Discussion of results}
We now make a number of remarks on our main results, the first of which is the following special case of \thref{stronger}.
\begin{corollary}\thlabel{aymone}
Let $g: \N \ra \U$ be multiplicative, and let $\mu_k$ be the $k$-free indicator function, defined as
\[\mu_k\bra{p^j} = \begin{cases}
    1 & \text{if } j < k,\\
    0 & \text{if } j \geq k.
\end{cases}\]
Then, for any $\al \in \R$,
\[\sup\limits_{x}\abs{\sum\limits_{n \leq x}g(n)\mu_k(n)\e(n\al)} = \infty.\]

\end{corollary}
\begin{remark}
Recent work of Aymone \cite{Aymone} proved \thref{aymone} in the case when $\al = 0$, $g: \N \ra \{-1,+1\}$ is completely multiplicative and $k=2,3$. He remarks that proving the result for all $k$ is possible with the methods in his paper, but that extending this to multiplicative functions is the subject of future research. Our result allows us to deal with an even wider class of functions, as well as showing that the result holds even if an exponential twist is included in the sum.
\end{remark}

\begin{remark}
    If we remove the assumption that $\al$ is irrational in \thref{special}, we can still say something stronger than \thref{stronger} when $f$ is completely multiplicative, takes finitely many values, or both. In particular, it follows from the rotation trick used to prove \thref{stronger} that if $f$ is completely multiplicative, $f(p) = \chi(p)p^{it}$ for all but finitely many primes $p$, and if $f$ takes finitely many values, then for all but finitely many primes $p$, we must have that $f(p^k)=\chi(p^k)$ for all $k \geq 1$.
\end{remark}

\begin{remark} 
    It is easy to show (see \thref{character2} below) that if $f(n) = \chi(n) n^{it}$ for some $\chi$ and $t,$ then $S_f(x,\al)$ is bounded for all irrational $\alpha.$ \thref{special} i) then shows this is in fact the only case where the exponential sum can be bounded for some irrational $\alpha$ when $f$ is completely multiplicative, thus giving a necessary and sufficient condition for boundedness in this case.
\end{remark}

We also mention the following conjecture, communicated to the second author by C{\'e}dric Pilatte, based on a recent paper of Fregoli \cite{fregoli2021note}. We note that \thref{special} ii) proves this conjecture in the special case where $f$ is multiplicative.
\begin{conjecture}\thlabel{conjecture}
For a function $f: \N \ra \{0,1\}$ and $E \sub \T$ with positive measure, the exponential sum $S_f(x,\al)$ is bounded uniformly for all $x > 0$ and $\al \in E$ if and only if $f$ is eventually periodic.
\end{conjecture}
\begin{remark}
    For discussion of the optimality of our results, in particular \thref{stronger,bazinmain}, see Section \ref{5}, where we show that there exist multiplicative functions which take values on the unit circle and have bounded exponential sums, but $\left\{p:f(p) \neq \chi(p)p^{it}\right\}$ is infinite for any character $\chi$ and $t \in \R$.
\end{remark}

\begin{remark}\thlabel{remark1.6}
   We note that the condition
   \[\sum\limits_{p:f(p) = 0}\frac{1}{p} < \infty\]
   in \thref{stronger,special,bazinmain} is equivalent to the condition
    \[\sum\limits_{p^k : f\bra{p^k} = 0}\frac{1}{p^k} < \infty,\]
    since $\sum_{p}\sum_{k \geq 2}\frac{1}{p^k} < \infty$, and simply ensures that a positive proportion of the values of $f$ are non-zero. This is a fairly natural condition in this context, and appears e.g. in the refinement of Chudakov's conjecture (Corollary 1.5b) in \cite{klurman2021multiplicative}.
\end{remark}

\subsection{Overview of the paper}
\begin{itemize}
    \item In Section \ref{pretentious}, we show that $S_f(x,\al)$ being bounded implies that $f$ is of a relatively special form, namely that $f$ is `pretentious'. Roughly speaking, this means that $f$ `pretends to be' a Dirichlet character, i.e, a periodic function. Tao's logarithmically averaged Elliott result from \cite{tao2016logarithmically} will play a crucial role in this section, in addition to the language of `pretentious number theory' of Granville-Soundararajan.
    \item In Section \ref{3}, we prove \thref{stronger}, by combining the reduction from the previous section with sieve-theoretic, combinatorial, and analytic ideas based on the `rotation trick' introduced by Klurman, Mangerel, Pohoata and Ter{\"a}v{\"a}inen in \cite{klurman2021multiplicative} to study similar questions about partial sums of completely multiplicative functions. In particular, we develop a version of the rotation trick which allows us to extend their methods to handle general multiplicative functions, rather than just completely multiplicative functions, and to deal with $f$ having zeroes.
    \item In Section \ref{4}, we prove \thref{special,bazinmain}, by combining the reduction from Section \ref{pretentious} with the large-sieve based lower bound on exponential sums obtained by the first author in \cite{bazin2025exponentialsums}. We note that the Duffin--Schaeffer--Koukoulopoulos--Maynard theorem \cite{duffin1941khintchine, koukoulopoulos2020duffinschaeffer} plays a key role in the proof of \thref{bazinmain} in order to find suitable rational approximations of $\alpha$ for almost all $\alpha.$ 
    \item In Section \ref{5}, we discuss the sharpness of our main results and whether it would be possible to strengthen them. In particular, we prove results showing that a statement of the strength of \thref{special} does not hold in the general case, so \thref{stronger,bazinmain} are in some sense close to optimal.
\end{itemize}

\section{Reduction to pretentious functions} \label{pretentious}

In this section, we show that being close to periodic is, in some sense, the only way $S_f(x,\al)$ can be bounded. We will show that, if $S_f(x,\al)$ is bounded in $x$ for some $\al$, then $f$ must `look like' a Dirichlet character. We will use the language of \emph{pretentious functions}, as developed by Granville-Soundararajan \cite{granville2014multiplicative}, to make precise what we mean later in this section.
\subsection{Dirichlet characters}
We begin this section by showing that the exponential partial sums of a `twisted' Dirichlet character are bounded, unless $\al$ is a rational number whose denominator is a divisor of the period of the Dirichlet character. Before we prove this, we need an additional lemma.
\begin{lem}\thlabel{character2_sums}
    Let $z \in \U$, $z\neq 1$ and let $t\in\R$. Then for $m \ge 1,$
    \begin{align*}
        \sum_{n=1}^m z^n n^{it}&=O_t\bra{\frac{1}{\abs{1-z}}}.
    \end{align*}
\end{lem}
\begin{proof}
We first show that for integers $m \ge A \ge 1,$ 
\begin{equation}\label{step1}
    \sum_{n=A}^m \frac{z^n n^{it}}{n}=O_t\bra{\frac{1}{A\abs{1-z}}}.\tag{*}
\end{equation} 
We have
\begin{align*}
    \abs{(1-z)\sum_{n=A}^m \frac{z^n n^{it}}{n}} &= \abs{-z^{m+1}m^{it-1}+z^{A}A^{it-1}+\sum_{n=A+1}^m z^n (n^{it-1}-(n-1)^{it-1})}
    \\&\le \frac{1}{m} + \frac{1}{A} + \sum_{n = A+1}^m O_t\bra{\frac{1}{n^2}} = O_t\bra{\frac{1}{A}},
\end{align*}
as desired after dividing by $1 - z \ne 0.$
\smallbreak
Now we show that $\sum_{n=1}^m z^n n^{it}=O_t\bra{\frac{1}{\abs{1-z}}}.$ For $m\ge A\ge 1,$ we have
\begin{align*}
    (1-z)\sum_{n=A}^m z^n n^{it} &= -z^{m+1}m^{it}+z^A A^{it} +\sum_{n=A+1}^m z^n (n^{it}-(n-1)^{it})
    \\&= O(1) + \sum_{n=A+1}^m z^n \bra{it n^{it-1} + O_t(n^{-2})}
    \\&= O_t(1) + it \sum_{n=A+1}^m z^n n^{it-1}.
\end{align*}
By \eqref{step1},we then have 
\begin{align*}
    \sum_{n=1}^m z^n n^{it} &= \sum_{n=1}^{A-1} z^n n^{it} + \frac{1}{1-z}\bra{O_t(1) + it \sum_{n=A+1}^m z^n n^{it-1}}
    \\&= O(A) + O_t\bra{\frac{1}{|1-z|}} + O_t\bra{\frac{1}{A\abs{1-z}^2}},
\end{align*}
hence the conclusion picking $A = \left\lceil\frac{1}{|1-z|}\right\rceil.$
\end{proof}

Now we prove the aforementioned result on twisted characters having bounded exponential sums.
\begin{prop}\thlabel{character2}
    Let $\chi$ be a Dirichlet character of modulus $m$, $t\in\R$ and let $\al \in \R$ be such that $m\al \notin \Z$. Then
    \[\sum\limits_{n \leq x}\chi(n)\e(n\al)n^{it} = O_{m, t}\left(\frac{1}{\e(\alpha m) - 1}\right).\]
\end{prop}
\begin{proof}[Proof of \thref{character2}]
As $\chi$ is $m$-periodic, we can write $$\chi(n) = \sum_{r \bmod m} a_r \e\bra{n\frac{r}{m}},$$ where $a_r := \frac{1}{m} \sum_{n \bmod m} \chi(n) \e\bra{-n\frac{r}{m}} = O_m(1).$ We then have 
\begin{align*}
    \sum\limits_{n \leq x}\chi(n)\e(n\al)n^{it} &=
    \sum_{r\bmod m} a_r \sum\limits_{n \leq x}\e\bra{n\bra{\al + \frac{r}{m}}} n^{it} 
    \\&= \sum_{r\bmod m} a_r O_t\bra{\frac{1}{\abs{\e\bra{\al + \frac{r}{m}} - 1}}}
    \\&= O_{m, t}\left(\frac{1}{\e(\alpha m) - 1}\right),
\end{align*}
by \thref{character2_sums} and noting that $\abs{\e\bra{\al + \frac{r}{m}} - 1} \ge \frac{1}{m} \abs{\e\bra{\al m} - 1}$ for all $r.$

\end{proof}

\subsection{Pretentious functions}
\begin{defn}
Following Granville and Soundararajan \cite{granville2014multiplicative}, we define the \emph{pretentious distance} between two multiplicative functions $f$ and $g$ as
\[\D^2(f,g;y,x) = \sum\limits_{y < p \leq x} \frac{1 - \mathrm{Re}\bra{f(p)\conj{g(p)}}}{p}.\]
We write $\D(f,g;x) = \D(f,g;1,x)$, and say that $f$ `pretends to be like' $g$, or $f$ is \emph{$g$-pretentious} if $\D(f,g;\infty) < \infty$.
\end{defn}
\begin{remark}
Perhaps the most important property of the pretentious distance is that it satisfies the triangle inequality, that is, $\D\bra{f,g;x} + \D\bra{g,h;x} \geq \D\bra{f,h;x}$.
\end{remark}

We will use Tao's two-point correlation result from \cite{tao2016logarithmically}, which shows that a multiplicative function with high self-correlation must resemble $\chi(n)n^{it}$ for some Dirichlet character $\chi$ and $t \in \R$. For our purposes, the following (slightly simplified) version of Theorem 1.3 from \cite{tao2016logarithmically} is sufficient.

\begin{thm}[Tao]\thlabel{Tao}
Let $\ep > 0$, $h \in \N$ and let $A\bra{\ep,h}$ be large. Let $f: \N \ra \C$ be a multiplicative function with $\abs{f(n)} \leq 1$ for all $n$ such that $\D\bra{f,\chi(n)n^{it};x} \geq A$ for all Dirichlet characters $\chi$ with period at most $A$, and all $t \in \R$ with $\abs{t} \leq Ax$. Then
\[\abs{\sum\limits_{n \leq x}\frac{f(n) \conj{f\bra{n+h}}}{n}} \leq \ep \log x.\]
\end{thm}

We also need the following technical result, which is Lemma 4.2 in \cite{klurman2017correlations}. This will allow us to iterate \thref{Tao} on a suitable set of scales.

\begin{lem}[Klurman]\thlabel{Klurman}
Let $x_1, x_2, \ldots$ be an infinite sequence such that $x_n < x_{n+1} \leq x_n^a$ for all $n$, for some $a > 1$. Assume that for each $x_m$ there exists a primitive Dirichlet character $\chi_m$ and a real number $t_m$ such that $\D\bra{f, \chi_m(n)n^{it_m};x_m} = O(1)$, where $\chi_m$ has conductor $O(1)$ and $\abs{t_m} \ll x_m$. Then there exists a primitive Dirichlet character $\chi$ and a real number $t \in \R$ such that $\D\bra{f, \chi(n)n^{it};\infty} < \infty$.
\end{lem}

Combining \thref{Tao,Klurman}, we have the following corollary, which we will use in our reduction.
\begin{corollary}
\thlabel{Elliott}
Let $f: \N \ra \C$ be a multiplicative function with $\abs{f(n)} \leq 1$ for all $n$, such that there exists some $H \ge 1$ such that for $x \ge 2,$
\[\max_{1\le h\le H} \abs{\sum\limits_{n \leq x}\frac{f(n)\conj{f(n+h)}}{n}} \gg \log x.\]
Then there is a Dirichlet character $\chi$ and a real number $t$ such that $\D\bra{f,\chi(n)n^{it};\infty} < \infty$.
\end{corollary}
\begin{proof}
    By the contrapositive of \thref{Tao}, for all $x\ge 2,$ if
    \[\abs{\sum\limits_{n \leq x}\frac{f(n)\conj{f(n+h)}}{n}} \ge  \ep \log x\]
     for some $1 \le h\le H$ (where $h$ may depend on $x$ but not $H$), then there is some $A(\ep,H) > 0$ such that, for sufficiently large $x$, there is a real number $t_x$ with $\abs{t_x} \leq Ax$ and a primitive character $\chi_x$ of modulus $\leq A$ such that $\D\bra{f,\chi_x(n)n^{it_x};x} \leq A$. Since $A$ does not depend on $x$, we can apply \thref{Klurman}, and the result follows immediately.
\end{proof}

\subsection{Performing the reduction}
We now apply \thref{Elliott} to our problem. The following lemma uses a `van der Corput' argument to show that any function with bounded partial sums, which is not too small, must exhibit some logarithmic correlation.

\begin{lemma}\thlabel{cor}
Let $f: \N \ra \C$, and assume $\sum_{n\le x} \frac{\abs{f(n)}^2}{n} \gg \log x.$ 
If $f$ has bounded partial sums, meaning
\[\sum\limits_{n \leq x} f(n) = O(1),\]
then there exists $H > 0$ such that for $x\ge 2$ we have
\[\max_{1\le h\le H} \abs{\frac{1}{\log x} \sum\limits_{n \leq x} \frac{f(n) \conj{f(n+h)} }{n}}\gg 1.\]
\end{lemma}
\begin{proof}
If $f$ has bounded partial sums, then certainly for any $H>0$ the sums 
$\sum\limits_{n \leq m \leq n+H} f(m)$ 
are bounded for any $n$. 
Therefore, there is some absolute constant $C$ such that
\[\abs{\sum\limits_{n \leq m \leq n+H} f(m)}^2 \leq C.\]
But if we expand out the sum, we have
\begin{align*}
    1 \gg \abs{\sum\limits_{n \leq m \leq n+H} f(m)}^2 &= \sum\limits_{0\le h_1, h_2 \leq H}f(n+h_1)\conj{f(n+h_2)}
    \\ &=\sum\limits_{0\le h_1 \leq H}\abs{f(n+h_1)}^2 + \sum\limits_{0\le h_1 \neq h_2 \leq H}f(n+h_1)\conj{f(n+h_2)}.
\end{align*}
Taking the logarithmic average, we then have for $x$ large enough in terms of $H,$
\begin{align*}
    1&\gg \frac{1}{\log x} \sum_{n\le x} \frac{1}{n} \abs{\sum\limits_{n \leq m \leq n+H} f(m)}^2 \\&=  \frac{1}{\log x} \sum_{n\le x} \frac{1}{n} \bra{\sum\limits_{0\le h_1 \leq H}\abs{f(n+h_1)}^2 + \sum\limits_{0\le h_1 \neq h_2 \leq H}f(n+h_1)\conj{f(n+h_2)}}
    \\&= \frac{1}{\log x} \sum_{n\le x+H} \abs{f(n)}^2 \sum_{\substack{0\le h\le H \\ 0\le n-h \le x}} \frac{1}{n-h} + \frac{1}{\log x} \sum_{1\le |h|\le H} \sum_{n\le x+H} f(n)\conj{f(n+h)} \sum_{\substack{0\le h_1,h_2\le H \\ h_2 - h_1 = h \\ 0\le n-h_1 \le x}} \frac{1}{n-h_1}
    \\&\asymp \frac{H+1}{\log x}  \sum_{n\le x} \frac{\abs{f(n)}^2}{n} + \sum_{1\le |h| \le H} \frac{H+1-|h|}{\log x} \sum_{n\le x} \frac{f(n) \conj{f(n+h)}}{n} 
    \\&\gg H + O\bra{H^2 \max_{1\le h\le H} \abs{\frac{1}{\log x} \sum_{n\le x} \frac{f(n) \conj{f(n+h)}}{n}}},
\end{align*}
where the implicit constants do not depend on $H$ (only on the implicit constants in the assumptions on $f$). Clearly, this is only possible if 
\[\max_{1\le h\le H} \abs{\frac{1}{\log x} \sum_{n\le x} \frac{f(n) \conj{f(n+h)}}{n}} \gg \frac{1}{H} + O\bra{\frac{1}{H^2}},\] which gives the desired conclusion after fixing some large enough $H$ (which depends on the implicit constants above but not on $x$).

\end{proof}

Using \thref{cor} and \thref{Elliott}, we can now reduce our problem to the class of functions which `pretend' to be a Dirichlet character.
\begin{prop} \thlabel{new1}Let $\al \in \R$, and let $f$ be a multiplicative function which satisfies the hypotheses of \thref{stronger}. If
\[\sup\limits_{x}\abs{S_f(x,\al)} < \infty,\]
then $\D(f, \chi(n) n^{it};\infty) < \infty$ for some Dirichlet character $\chi$ and $t \in \R$.
\end{prop}

\begin{proof}
    Since $f$ is multiplicative and takes values in $\U \cup \{0\},$ we have 
    \[\sum_{n\le x} \frac{\abs{f(n)\e(\al n)}^2}{n} = \sum_{\substack{n\le x \\ f(n)\ne 0}} \frac{1}{n} \gg \log x\] as $\sum_{p:f(p)\ne 0} \frac{1}{p} < \infty.$
    \smallbreak
    We can therefore apply \thref{cor} with $f(n)\e(n\alpha)$, so that for some $H > 0$, we have for $x\ge 2$
    \[\max_{1\le h\le H} \abs{\frac{1}{\log x}\sum\limits_{n \leq x}\frac{1}{n}f(n)\e(n\al) \conj{f(n+h)\e((n+h)\al)}} \gg 1.\]
    Since $\e\bra{n\al}\conj{\e\bra{\bra{n+h}\al}} = \e\bra{-h\al}$ has modulus one and does not depend on $n,$ this reduces to 
    \[\max_{1\le h\le H} \abs{\frac{1}{\log x}\sum\limits_{n \leq x}\frac{f(n)\conj{f(n+h)}}{n}} \gg 1.\]
    Applying \thref{Elliott}, we deduce that indeed we have $\D(f, \chi(n)n^{it};\infty) < \infty$ for some Dirichlet character $\chi$ and real number $t$.
\end{proof}

\section{The rotation trick}\label{3}
In this section, we strengthen the conclusion of \thref{new1} by proving \thref{stronger}. Throughout this section, we assume that $f$ satisfies the hypotheses of \thref{stronger}, namely that $f: \N \ra \C$ is multiplicative, for all $n \in \N$, $\abs{f(n)} \in \{0,1\}$, and $\sum_{p^k \in T}\frac{1}{p^k} < \infty$, where 
\[T =\{p^k:f\bra{p^k} = 0 \text{ for some } k\}.\] 
By \thref{new1}, if
\[\sup\limits_{x}\abs{S_f(x,\al)} < \infty\]
for some $\al \in \R$, then there is some Dirichlet character $\chi$ and $t \in \R$ such that 
\[\D\bra{f,\chi(n)n^{it};\infty} < \infty.\]
We will assume throughout this section that $f$ is $\chi(n)n^{it}$-pretentious.

\subsection{Lemmata for the rotation trick}

We will need the following concentration inequality for pretentious functions, based on the Tur\'an-Kubilius inequality, which appears as Lemma~2.5 in \cite{klurman2021multiplicative}. For brevity, we will use $F(Q)$ to denote the quantity
\[F(Q) = \sum\limits_{\substack{p \leq x \\ p \nmid Q}}\frac{f(p)\conj{\chi}(p)p^{-it}-1}{p}.\]

\begin{lemma}[Concentration inequality for pretentious functions]\thlabel{concentration}
    Let $f: \N \ra \C$ be a multiplicative function with $\abs{f(n)} \leq 1$ for all $n \in \N$, such that $\D\bra{f,\chi(n)n^{it}; \infty} < \infty$ for some Dirichlet character $\chi$ of conductor $q$ and $t \in \R$. Let $Q, N \geq 1$ be integers with $\prod_{p \leq N}p \mid Q$ and $q \mid Q$. Then for any $1 \leq a \leq Q$, with $\bra{a,Q} = 1$, we have
    \[\sum\limits_{n \leq x}\abs{f(Qn+a) - \chi(a)\bra{Qn}^{it} \exp\bra{F(Q)}} \ll x\bra{\D\bra{1, f(n)\conj{\chi}(n)n^{-it};N,x} + \frac{1}{\sqrt{N}}}.\]
\end{lemma}
\begin{remark}
    This inequality is necessary to deal with the case where $\{f(n) : n \in \N\}$ is infinite. If $f$ takes only finitely many values, then it is not hard to show that, if we let $S = \{p : f\bra{p^k} \neq \chi\bra{p^k}p^{ikt} \text{ for some } k\}$, then
    \[\sum\limits_{p}\frac{1-\Re \bra{f(p)\conj{\chi}(p)p^{-it}}}{p} < \infty \implies \sum\limits_{p \in S}\frac{1}{p} < \infty,\]
    which allows us to `sieve out' the primes in $S$. However, in the general case this is no longer true, and we must instead replace this with the concentration inequality, \thref{concentration}.
\end{remark}




\subsection{Applying the rotation trick}
With the relevant lemmata established, we can now move to proving \thref{stronger}. Much of this proof is based on the proof of Theorem 1.4 in \cite{klurman2021multiplicative}, but with several significant modifications. In particular, we need to deal with the extra exponential twists in the sums, and the fact we are working with multiplicative functions, rather than completely multiplicative functions, as well as some extra work to control the zeroes of $f$.
\bigbreak
We will first prove the following weaker version of \thref{stronger}. The only difference here is with the conclusion - \thref{stronger} says that the sum over $p$ of the maximum over $k$ of $\abs{f(p^k)-\chi(p^k)p^{ikt}}$ converges, but first we show that these individual terms tend to zero, except for possibly when $f\bra{p^k}=0$. We will then use this to prove the full \thref{stronger}.

\begin{prop}\thlabel{weaker}
    Let $f: \N \ra \U \cup \{0\}$ be a multiplicative function, and assume that
    \[\sum_{p:f(p) = 0}\frac{1}{p} < \infty.\]
   If, for some $\al \in \R$,
    \[\sup\limits_{x}\abs{S_f(x,\al)} < \infty,\]
    then there is a Dirichlet character $\chi$ and a real number $t \in \R$ such that
    \[\lim_{p\ra\infty}\max_{k: f\bra{p^k} \neq 0}\abs{f(p^k)-\chi(p^k)p^{ikt}} = 0.\]

\end{prop}

\begin{proof}
Let $\ep$, $H$, $w$ and $x$ be parameters satisfying
\[1 \ll \frac{1}{\ep} \ll H \ll w \ll x,\]
and we let $W = \prod_{p \leq w}p^w$. We also fix $k_1,\ldots, k_H$ (to be chosen later) such that the $k_j$ are pairwise coprime and all coprime to $W.$ 
\bigbreak
With $P := \prod_{1\le j \le H} k_j^2,$ it then follows from the Chinese remainder theorem that we can find $n_0 \le P$ such that 
\[Wn_0 = k_j - j \pmod{k_j^2}\] for all $j.$ In particular, when $n = n_0 + mP$ for some $m$, we have for all $j,$
\[jk_j \|\bra{Wn+j},\]
so that
\begin{align*}
    f(Wn + j) &= f\big(W(n_0 + mP) + j\big)
    \\&= f(j) f(k_j)f\bra{\frac{W n_0}{j k_j} + \frac{WmP}{j k_j} + \frac{1}{k_j}}
    \\&= f(j) f(k_j)f\bra{\frac{WmP}{jk_j} + \frac{\frac{W n_0}{j} + 1}{{k_j}}}.
\end{align*}

\bigbreak

Applying \thref{concentration}, with \[Q = \frac{WP}{jk_j},\, N = w,\, a= \frac{\frac{W n_0}{j} + 1}{{k_j}},\]
noting that by definition of $F$ we have $F(Q) = F(WP)$ for all $j,$ and that by construction of $n_0$ we have $a$ coprime with all the $k_{j'}$ and therefore with $Q,$
and picking $w$ large enough such that $\D\bra{f, \chi(n)n^{it};w,\infty} \leq 1/H,$ we have 
\begin{equation}\label{eq1}
\frac{1}{x}\sum\limits_{m \leq x}\abs{f\bra{\frac{WmP}{jk_j} + \frac{\frac{W n_0}{j} + 1}{{k_j}}} - \chi\bra{\frac{\frac{W n_0}{j} + 1}{{k_j}}}\bra{\frac{WmP}{jk_j}}^{it} \exp\bra{F(WP)}} \ll \frac{1}{H}.
\end{equation}

By construction, with $q$ the conductor of $\chi$ and choosing $w \ge q,$
\[\frac{\frac{W n_0}{j} + 1}{{k_j}} = k_j^{-1} \pmod{q},\]
and thus
\[\chi\bra{\frac{\frac{W n_0}{j} + 1}{{k_j}}} = \chi\bra{k_j^{-1}} = \conj{\chi}(k_j).\]

\smallbreak
Therefore, after multiplying by $f(j)f(k_j) \e(j \al)$, which has absolute value $0$ or $1$, the summands in the left hand side of \eqref{eq1} can be written as

\[f(j)f(k_j)\e(j\al)f\bra{\frac{WmP}{jk_j} + \frac{\frac{W n_0}{j} + 1}{{k_j}}} -f(j)f(k_j)\e(j\al)\conj{\chi}(k_j)k_j^{-it}\bra{\frac{WmP}{j}}^{it}\exp\bra{F(W)}.\]
Since from above we have 
\[f(j)f(k_j)f\bra{\frac{WmP}{jk_j} + \frac{\frac{W n_0}{j} + 1}{{k_j}}} = f(Wn+j),\]
we see that \eqref{eq1} can be written as

\[\frac{1}{x} \sum\limits_{m \leq x}\abs{f(Wn+j)\e(j\al) - \e(j\al)f(j)f(k_j) \conj{\chi}(k_j)k_j^{-it}\bra{\frac{WmP}{j} }^{it}\exp\bra{F(WP)}} \ll \frac{1}{H},\]
where $n = n_0 + mP$. Summing over $1 \le j \le H$ and using the triangle inequality, we then have
\begin{align*}
    2\sup\limits_{n\le WP(x+2)} \abs{S_f(n,\al)} &\ge \frac{1}{x} \sum\limits_{m\le x} \abs{\sum_{1\le j \le H} f(Wn+j) \e(j\al)} \\&= \frac{1}{x} \sum\limits_{m\le x} \abs{(WmP)^{it} \exp\bra{F(WP)} \sum_{1\le j \le H} \e(j\al) f(j) f(k_j) \conj{\chi}(k_j) k_j^{-it} j^{-it}} + O(1) 
    \\&= \abs{\exp\bra{F(WP)}} \abs{\sum_{1\le j \le H} \e(j\al) f(j) f(k_j) j^{-it} \conj{\chi}(k_j) k_j^{-it}} + O(1)
    \\&= \abs{\sum_{1\le j \le H} \e(j\al) f(j) f(k_j) j^{-it} \conj{\chi}(k_j) k_j^{-it}} + O(1)
\end{align*}
noting that $\abs{\Re\big(F(WP)\big)} \le \D^2\bra{f, \chi(n)n^{it};w,\infty} \le 1/H$ when $w$ is large enough.
\smallbreak

Now note that if $S$ is finite, then certainly
\[\lim\limits_{p \ra \infty}\max_k \abs{f(p^k) - \chi(p^k)p^{ikt}} \ra 0,\]
since $f(p^k)=\chi(p^k)p^{ikt}$ for all large enough $p$. So we assume that $S$ is infinite, and show how we can choose the $k_j$ to make 
\[\abs{\sum\limits_{1 \leq j \leq H}\e(j\al)f(j)j^{-it}f(k_j)\conj{\chi}(k_j)k_j^{-it}}\]
large, and hence derive a contradiction, unless $\lim\limits_{p \ra \infty}\max_{k: f(p^k) \neq 0} \abs{f(p^k) - \chi(p^k)p^{ikt}} \ra 0$.
\bigbreak
For $j$ such that $f(j)\ne 0,$ let $\be_j$ be such that $\e(\be_j) = \e(j\al)f(j)j^{-it}$, and let $g(n) = f(n)\conj{\chi}(n)n^{-it}$. We note that, as a product of multiplicative functions, $g$ is itself multiplicative. We now show how we can construct $k_j$ such that $g(k_j) \approx \e(-\be_j),$ so that \[\sum\limits_{1 \leq j \leq H}\e(j\al)f(j)j^{-it}f(k_j)\conj{\chi}(k_j)k_j^{-it} = \sum_{\substack{1\le j \le H \\ f(j)\ne 0}} \e(\be_j) g(k_j) \gg H.\]

\smallbreak
Since by assumption there are infinitely many prime powers $p^{\ell}$ such that $g\bra{p^{\ell}} \neq 0, 1$, there must be some sequence $p_m^{\ell_m}$ such that $g\bra{p_m^{{\ell}_m}} \ra \e\bra{\ph}$ for some $ 0 \leq \ph < 1$, where each $g\bra{p_m^{\ell_m}}$ is a complex number of absolute value $1$, but not equal to $1$. We need to consider separately the cases where $\ph$ is rational or irrational.
\bigbreak
In the case where $\ph$ is irrational, since $\{\e(r\ph) : r \in \N\}$ is equidistributed on the unit circle, there is some integer $r$ such that 
\[\abs{\e\bra{r\ph}\e\bra{\be_j}-1} \leq \frac{1}{100}.\]
Since $g\bra{p_m^{\ell_m}} \ra \e\bra{\ph}$, we can find $N \gg w$ such that $\abs{g\bra{p_m^{\ell_m}}-\e\bra{\ph}} < \frac{1}{100r}$ for every $N < m \leq N+r$, and hence setting $k_j = \prod\limits_{N < m \leq N+r} p_m^{\ell_m}$, we have
\begin{align*}
    \e\bra{\be_j}g\bra{k_j} &= \e\bra{\be_j}g\bra{ \prod\limits_{N < m \leq N+r} p_m^{\ell_m}}
    \\&= \e\bra{\be_j}\prod\limits_{N < m \leq N+r} g\bra{p_m^{\ell_m}},
\end{align*}
and therefore
\begin{align*}
    \abs{\e\bra{\be_j}g\bra{k_j} - 1} &= \abs{\e\bra{\be_j}\prod\limits_{N < m \leq N+r} g\bra{p_m^{\ell_m}}-1}
    \\&\leq \abs{\e\bra{\be_j}\e\bra{r\ph}-1} + \frac{1}{100}
    \\&\leq \frac{1}{50}.
\end{align*}

We can do this for each $\be_j$, making sure we choose different primes each time (which is possible by assumption), so that
\[\sum_{\substack{1\le j \le H \\ f(j)\ne 0}} \e(\be_j) g(k_j) \gg H.\]
\bigbreak
In the case where $\ph$ is rational, but $\ph \neq 0, \frac{1}{2}$, by the pigeonhole principle there is some $r$ such that $\abs{\e(r\ph)\e(\be_j)-1} \leq \frac{1}{3}$, and the argument is essentially the same as above, only possibly with a worse implicit constant, and so in this case we also have
\[\sum_{\substack{1\le j \le H \\ f(j)\ne 0}} \e(\be_j) g(k_j) \gg H.\]
\bigbreak
When $\ph = \frac{1}{2}$, we can do something similar since a positive proportion of the $\e(\be_j)$ are not equal to $\pm i$, by using the primes to rotate by $-1$.
\bigbreak
Thus the only remaining case is where, for every convergent sequence $p_m^{\ell_m}$, we have $g(p_m^{\ell_m}) \ra 1$. It is not too hard to show that this is equivalent to the statement that
\[\lim\limits_{p \ra \infty} \max_{k: f(p^k) \neq 0} \abs{g(p^k)-1} = 0.\]
Indeed, suppose $\max_{k:f(p^k) \neq 0} \abs{g(p^k)-1}$ does not tend to $0$ as $p \ra \infty$. Then there is $\ep>0$ and a sequence $p_i$ such that for each $p_i$ there is $k_i$ such that
\[\abs{g(p_i^{k_i}) - 1} > \ep,\]
and hence there must be a convergent subsequence of $\{g(p_i^{k_i})\}$, whose limit is necessarily at least $\ep$ away from $1$, but is not $0$. Recalling that $g(n)=f(n)\conj{\chi}(n)n^{-it}$, we have
\[\lim\limits_{p \ra \infty} \max_{k: f(p^k) \neq 0} \abs{g(p^k)-1} = 0 \iff \lim\limits_{p \ra \infty} \max_{k: f(p^k) \neq 0} \abs{f(p^k)-\chi(p^k)p^{ikt}} = 0,\]
proving \thref{weaker}.
\end{proof}

We now show that, using another iteration of the rotation trick, one can remove the condition in \thref{weaker} that $f(p^k) \neq 0$, by showing that in fact there are only finitely many primes $p$ where $f(p^k)$ is zero.

\begin{lem}\thlabel{zeroes}
Let $f: \N \ra \U \cup \{0\}$ be a multiplicative function, and assume that
    \[\sum_{p:f(p) = 0}\frac{1}{p} < \infty.\]
   If, for some $\al \in \R$,
    \[\sup\limits_{x}\abs{S_f(x,\al)} < \infty,\]
    then the set
    \[\{p : f(p^k) = 0 \text{ for some }k\}\]
    is finite.
\end{lem}

\begin{proof}
    By \thref{weaker}, there is some $\chi$ and $t$ such that, if $g(n) = f(n)\conj{\chi}(n)n^{-it}$,
\[\lim_{p\to\infty}\max_{k:f(p^{k})\ne0}|g(p^{k})-1|=0.\]
We assume for a contradiction that
\[P_0 = \{p:f(p^k)=0 \text{ for some }k\}\]
is infinite. The method here will be similar to how we used certain prime powers to `rotate' and force the sum to be large in the proof of \thref{weaker} - in this case, we destroy the cancellation in the sum by multiplying certain terms in the sum by zero. As before, we do this by making the sum

\[\left|\sum_{1 \le j \le H}f(Wn+j)\e(j\alpha)\right|\]
over a carefully chosen arithmetic progression large by showing that we can make

\[\left| \sum_{1\le j\le H} \e(\beta_j)g(k_j) \right|\]
grow with $H$, where $\e(\beta_j) = \e(j\alpha)f(j)j^{-it}$.
\bigbreak
First, we define 4 sets as follows:
\begin{align*}
J_{E} &= \{j : -\frac{1}{8} \le \be_j < \frac{1}{8}\},\\
J_{N} &= \{j : \frac{1}{8} \le \be_j < \frac{3}{8}\},\\  
J_{W} &= \{j : \frac{3}{8} \le \be_j < \frac{5}{8}\},\\   
J_{S} &= \{j : -\frac{3}{8} \le \be_j < -\frac{1}{8}\},\\   
\end{align*}
picking the $\be_j$ (defined modulo $1$) in $[-3/8, 5/8)$ so that $J_E, J_N, J_W, J_S$ partition the indices $1\le j\le H.$

Clearly one of the 4 sets contains at least $\frac{H}{4}$ of the indices $1 \leq j \leq H$ - let this set be $J_{good}$ and the union of the other 3 be $J_{bad}$. 
\smallbreak
For each $j$ in $J_{bad}$, select $p_j \in P_0$ such that $p_j > H$. Let $\ell_j$ be the power of $p_j$ for which $f(p_j^{\ell_j})=0$, and set $k_j = p_j^{\ell_j}$. It therefore follows that for each $j \in J_{bad}$, $g(k_j)=0$. For $j$ in $J_{good},$ we simply take $k_j = 1,$ so that $g(k_j) = 1.$

\smallbreak
As in the proof of \thref{weaker}, we can use the Chinese remainder theorem to guarantee some arithmetic progression $Wn+1, \ldots, Wn+H$ where each $k_j$ is an exact divisor of $Wn+j$ and is coprime to all other $Wn+j'$ for $j' \neq j$. Then, we have
\[\sum_{j=1}^{H} \e(\beta_j)g(k_j) = \sum\limits_{j \in J_{good}}\e(\be_j) \gg H.\]
Indeed, assume for example that $J_{good} = J_E$, then
\[\text{Re}\bra{\sum\limits_{j \in J_{good}}\e(\be_j)} \geq \frac{1}{\sqrt{2}}\abs{J_{good}} \geq \frac{H}{4\sqrt{2}},\]
and similarly for the other three cases, the real or imaginary part must be $\gg H$ in magnitude.
\smallbreak
But by the same reasoning as in the first iteration of the rotation trick in \thref{weaker}, we have a contradiction to $\sum\limits_{n \leq x}f(n) \e(n\al)$ being bounded, and hence we were wrong to assume $P_0$ was infinite. So, we conclude that there must be finitely many primes $p$ where $f(p^k) =0$ for some $k$.
\end{proof}

We can now strengthen \thref{weaker} to prove the full \thref{stronger}.
\begin{proof}[Proof of \thref{stronger}]
Assume that, with $g(n) = f(n)\conj{\chi}(n)n^{-it}$ as above,
$$\lim\limits_{\substack{p \ra \infty}} \max_{k} \abs{g(p^k)-1} = 0, \text{ but } \sum_{p} \max_{k} \abs{g(p^k)-1} = +\infty.$$ 
Then we can find for each $p > w$ some $m_p$ and $\theta_p$ such that 
\begin{align*}
\forall p, g(p^{m_p}) &= \e(\theta_p),\\
\lim\limits_{p \ra \infty} \theta_p &= 0,\\ 
\sum_{p > w} |\theta_p| &= +\infty,
\end{align*} since we have $\theta_p \asymp g(p^{m_p}) - 1$ when $g(p^{m_p})\ra 1.$ 
Without loss of generality, we may assume $$\sum_{\substack{p > w\\ \theta_p > 0}} \theta_p = +\infty$$ (the case where the sum over $\theta_p < 0$ goes to infinity is similar).
Now, pick $M_1$ such that for all $p \ge M_1, |\theta_p| < \frac{1}{20}$ and pick $N_1\ge M_1$ the smallest integer such that $$\sum_{\substack{M_1\le p\le N_1\\ \theta_p > 0}} \theta_p \ge 1-\beta_1.$$ Then, $M_1$ and $N_1$ verify
$$1 - \beta_j \le \sum_{\substack{M_1\le p\le N_1\\ \theta_p > 0}} \theta_p < 1-\beta_1 + \theta_{N_1} < 1 - \beta_1 + \frac{1}{20},$$ so that defining 
$$k_1 := \prod_{\substack{M_1\le p\le N_1\\ \theta_p > 0} }p^{m_p},$$ 
we have $$\abs{\e(\beta_1) g(k_1) - 1} = \abs{\e\bra{\beta_1 + \sum_{\substack{M_1\le p\le N_1\\ \theta_p > 0}} \theta_p} - 1} \le \frac{1}{20}.$$ We may not reuse the same primes to build the other $k_j,$ but the sum of $|\theta_p|$ over the leftover primes is still infinite as we only used a finite number of primes, allowing us to repeat the same argument for the other $j.$ By the same argument as in the proof of \thref{weaker}, this contradicts the assumption that $\sum f(n)\e(n\al)$ is bounded.
\smallbreak
This leaves us with the case where $$\sum_{p} \max_k \abs{g(p^k)-1} < \infty.$$ Recalling that $g(n)=f(n)\conj{\chi}(n)n^{-it}$, this is equivalent to the condition appearing in \thref{stronger}, which concludes the proof.
\end{proof}

\section{A result depending on $\alpha$}\label{4}
\subsection{Statement of results}

In this section, we prove the following result, which requires a weaker condition on $f$ but the result we obtain depends on $\alpha$. We will then use this to deduce \thref{special,bazinmain}. 
\smallbreak
As in Section \ref{3}, under the assumptions of \thref{setup-bazinmain}, we know $f$ is $\chi(n) n^{it}$-pretentious for some $\chi, t$ by \thref{new1}. We may furthermore assume that $\chi$ is primitive, as if $\chi$ is derived from a primitive character $\chi',$ then $f$ is also $\chi'(n) n^{it}$-pretentious. We then write
\begin{equation}\label{def-h}
    h(n) := \sum_{dk=n} \mu(d) f(k) \conj{\chi}(k) k^{-it},
\end{equation}
so that 
\begin{equation}
    f(n) = \chi(n) n^{it} \sum_{d|n} h(d). 
\end{equation}

\begin{thm}\thlabel{setup-bazinmain}
    Let $f: \N \ra \U \cup \{0\}$ be a multiplicative function, and assume that
    \[\sum_{p:f(p) = 0}\frac{1}{p} < \infty.\]
   If, for some $\al \in \R\smallsetminus\Q$,
    \[\sup\limits_{x}\abs{S_f(x,\al)} < \infty,\]
    then there is a Dirichlet character $\chi$ and a real number $t \in \R$ such that for $m$ the conductor of $\chi,$ $h$ defined by \eqref{def-h},
    $$\PP := \left\{p, \sum_{k\ge 0} \frac{h(p^k)}{p^k} = 0\right\},$$
    $P := \prod_{p\in\PP} p$,\footnote{Note that $\PP$ is always finite and in fact $\PP \subset \{2, 3\}.$ Indeed for all $k$ we have $\abs{h(p^k)} \le 2,$ so that $\abs{\sum_{k\ge 0} \frac{h(p^k)}{p^k}} \ge 1 - \sum_{k\ge 1} \frac{2}{p^k} > 0$ for $p > 3.$} and $g$ the multiplicative function defined by
    \[g(p^j) := \begin{cases}
         \chi(p^j) \frac{\sum_{k\ge 0} h\bra{p^{j+k + \Ind_{p\in\PP}}}p^{-k}}{\sum_{k\ge 0} h\bra{p^{k + \Ind_{p\in\PP}}}p^{-k}}, \text{ if } p\nmid m,
         \\ f(p^j) p^{-jit}, \text{ if } p|m
    \end{cases}\]
    (where the extra $\Ind_{p\in\PP}$ term ensures we are not dividing by $0,$ since $\sum_{k\ge 0} \frac{h\bra{p^{k + 1}}}{p^k} = -ph(1) = -p$ for $p\in\PP$), we have
    \[\sum_{\substack{a,n \\ (a,nPm) = 1}} \bra{\frac{|g(n)|}{n\abs{\alpha - \frac{a}{nPm}}}}^2 < \infty.\]
\end{thm}

The Duffin--Schaeffer theorem, proven by Koukoulopoulos and Maynard in 2020 \cite{koukoulopoulos2020duffinschaeffer}, tells us exactly for which $g$ the term $\frac{|g(n)|}{n\abs{\alpha - \frac{a}{nPm}}}$ is unbounded for almost all $\alpha.$ \footnote{Note that Duffin and Schaeffer in 1941 \cite{duffin1941khintchine} had already solved the case where $\sum_{n\le x} |g(n)| \frac{\phi(n)}{n} \gg \sum_{n\le x} |g(n)|,$ which is sufficient for our proof as $g$ is multiplicative.} As a consequence of it and \thref{setup-bazinmain}, we obtain a strengthening of \thref{stronger} when we assume the exponential sum is bounded for a positive measure set of $\alpha$, which is \thref{bazinmain}.

The proof of \thref{setup-bazinmain} is based on the following theorem of the first author \cite[Theorem 8]{bazin2025exponentialsums}, which provides a lower bound on exponential sums over short intervals when the exponential sum has major arcs at rationals.

\begin{thm}[Large sieve lower bound for exponential sums]\thlabel{large-sieve}
    Let $f$ supported on $[1,x].$ Then for $y \le x$ and $Q \le x^{1/2},$ we have 
    \begin{equation}\label{large-sieve-eq}
        \sum_{-y < n < x} \abs{\sum_{n < m \le n+y} f(m) \e(\alpha m)}^2 \gg \frac{1}{x} \sum_{\substack{q\le Q \\ (a,q) = 1}} \abs{S_f\bra{x,\frac{a}{q}}}^2 \abs{S_1\bra{y,\alpha - \frac{a}{q}}}^2,
    \end{equation}
    where $S_1(y,\be) = \sum_{n=1}^y \e(\be n).$ 
\end{thm}

When $f$ has bounded exponential sums, the left-hand side of \eqref{large-sieve-eq} is $\ll x+y.$ Moreover, when $f$ is pretentious (which we can assume by \thref{new1}), \thref{pretentious-majorarc} below shows $S_f\bra{x,\frac{a}{q}}$ is $\gg_q x$ when $f$ disagrees with $\chi(n)n^{it}$ at primes dividing $q,$ which we use to get a contradiction in \eqref{large-sieve-eq}. 

\subsection{Technical lemmas}

We need to quantify the size of $S_f\bra{x,\frac{a}{q}}$ at rationals in order to use \thref{large-sieve}. To do this, we use Hal\'asz's theorem as stated in \cite[Theorem III.4.5]{tenenbaum}.

\begin{lem}[Hal\'asz]\thlabel{halasz}
    Let $f$ be a multiplicative function taking values in the unit disc.
    \begin{itemize}
        \item If $\D(f, n^{it}, \infty) < \infty$ for some $t,$ then with $f(n) n^{-it} = 1 * h,$ we have as $x\to\infty$ \[\sum_{n\le x} f(n) 
        = \frac{x^{1+it}}{1+it} \prod_{p\le x}\sum_{k\ge 0} \frac{h(p^k)}{p^k} + o(x).\]
        \item Else, we have as $x\to\infty$ \[\sum_{n\le x} f(n) = o(x).\]
    \end{itemize}
\end{lem}

We then deduce an estimate of $\sum_{n\le x} f(dn)\conj{\chi_1}(n)$ for fixed $d$ and a character $\chi_1.$
\begin{lem}\thlabel{halasz-character}
    Let $f$ be a multiplicative function taking values in the unit disc. Assume $f$ is $\chi(n)n^{it}$-pretentious (where $\chi$ is primitive) and let $h(n) = \mu * f\conj{\chi} n^{-it}.$ Then for fixed $d, k$ and $\chi_1$ a character of conductor $k,$
    \begin{itemize}
        \item If $\chi_1$ is derived from $\chi,$ we have as $x\to\infty$ \[\sum_{n\le x} f(dn)\conj{\chi_1}(n) = \frac{x^{1+it}}{1+it} \frac{\ph(k)}{k} \prod_{p\le x} g_1(p) + o(x),\] where 
        $$g_1(p) =  \begin{cases}
            \sum_{r\ge 0} \frac{h(p^r)}{p^r}, &\text{ if } p \nmid dk, \\ 
            \frac{f(p^j)}{p^{j(1+it)}} + \chi(p^j) \sum_{r > j} \frac{h(p^{r})}{p^r}, &\text{ if } p^j\|d \text{ and } p\nmid k, \\
            \frac{f(p^j)}{p^{j(1+it)}}, &\text{ if } p|k \text{ and } p^j\|D \ (\text{allowing }j = 0).
    \end{cases}$$
        \item Else, we have as $x\to\infty$ \[\sum_{n\le x} f(dn) \conj{\chi_1}(n) = o(x).\]
    \end{itemize}
\end{lem}
\begin{proof}
    If $f$ is zero on multiples of $d,$ then $\sum_{n\le x/d} f(dn) \conj{\chi_1}(n) = 0.$ Otherwise, let $D$ the smallest multiple of $d$ such that $f(D) \ne 0$ (which has the same prime factors as $d$ with possibly higher multiplicities). Then, 
    \[\sum_{n\le x/d} f(dn) \conj{\chi_1}(n) = f(D) \conj{\chi_1}(D/d) \sum_{n\le x/D} \frac{f(Dn)}{f(D)} \conj{\chi_1}(n),\]
    where $\frac{f(Dn)}{f(D)} \conj{\chi_1}(n)$ is multiplicative and $\conj{\chi_1}(n)\chi(n) n^{it}$-pretentious. In particular, it is $n^{it}$-pretentious if and only if $m$ divides $k$ and $\chi_1$ is derived from $\chi.$ Thus the second case of the lemma follows from the second case of \thref{halasz}. We now assume $\chi_1$ is indeed derived from $\chi.$ In this case, we have $\frac{f(Dn)}{f(D)}\conj{\chi_1}(n)n^{-it} = 1 * h_1(n)$ with 
    $$h_1(p^r) =  \begin{cases}
        h(p^r), &\text{ if } p \nmid dk, \\ 
        -\Ind_{r=1}, &\text{ if } p | k, \\
        h(p^{r+j})\frac{\chi(p^j) p^{jit}}{f(p^j)}, &\text{ if } p^j\|D \text{ and } p\nmid k
    \end{cases}$$
    so that by \thref{halasz} we have when $x\to\infty$ while $D,\chi_1$ are fixed,
    \begin{align*}
        \sum_{n\le x/d} f(dn) \conj{\chi_1}(n) &= f(D) \conj{\chi_1}(D/d) \frac{(x/D)^{1+it}}{1+it} \prod_{p\le x/D} \sum_{r\ge 0} \frac{h_1(p^r)}{p^r} + o(x)
        \\&= \frac{x^{1+it}}{1+it} \prod_{p\le x} g_1(p) \frac{\ph(k)}{k} + o(x),
    \end{align*}
    where we could extend the product since $\sum_{x/D \le p \le x} \frac{1}{p} \ll \frac{1}{\log x}$ and we defined
    $$g_1(p) =  \begin{cases}
        \sum_{r\ge 0} \frac{h(p^r)}{p^r}, &\text{ if } p \nmid dk, \\ 
        f(p^J) \frac{\conj{\chi}(p^{J-j})}{p^{J(1+it)}} + \sum_{r\ge 1} \frac{h(p^{r+J}) \chi(p^j)}{p^{r+J}}, &\text{ if } p|d \text{ and } p\nmid k \text{ with } p^j\|d  \text{ and } p^J\|D, \\
        \frac{f(p^J)}{p^{J(1+it)}} \Ind_{p\nmid D/d}, &\text{ if } p|k \text{ and } p^J\|D \ (J \text{ might be } 0).
    \end{cases}$$
    Since $p$ divides $D/d$ if and only if $f(p^j) = 0$ for $p^j\|d,$ we have $g_1(p) = \frac{f(p^j)}{p^j}$ if $p|k$ and $p^j\|d.$ Moreover if $p|d$ and $p\nmid k,$ we have by definition of $h$ with $p^j\|d, p^J\|D$
    \begin{align*}
        g_1(p) &= \frac{\chi(p^j)}{p^J}\sum_{0\le r\le J} h(p^r) + \sum_{r\ge 1} \frac{h(p^{r+J}) \chi(p^j)}{p^{r+J}}
        \\&= \frac{f(p^j)}{p^Jp^{jit}} + \chi(p^j) \sum_{j < r \le J} \frac{h(p^r)}{p^J} + \chi(p^j) \sum_{r > J} \frac{h(p^r)}{p^r)}
        \\&= \frac{f(p^j)}{p^jp^{jit}} + \chi(p^j) \sum_{r > j} \frac{h(p^r)}{p^r)},
    \end{align*}
    noting that if $J > j$ we have $f(p^j) = 0$ and $h(p^r) = 0$ for $j < r < J.$
\end{proof}

We can now estimate $S_f\bra{x, \frac{a}{q}}$ when $q$ is a multiple of the conductor of $\chi.$ 

\begin{lem}\thlabel{pretentious-majorarc}
    Assume $f$ is $\chi(n)n^{it}$-pretentious, where $\chi$ is a primitive Dirichlet character of conductor $m.$ Define $P$, $h$ and $g$ as in \thref{setup-bazinmain}. Then for fixed $q$ and $a$ with $(a, qPm) = 1,$ we have when $x\to\infty$ $$S_f\bra{x, \frac{a}{qPm}} = \frac{x^{1+it}}{1+it} G(\chi) \conj{\chi}(a) \Pi(f, x) \frac{g(q)}{qPm} +o_q(x),$$ where we defined the Gauss sum $G(\chi) = \sum_{b\bmod m} \chi(b) \e\bra{\frac{b}{m}}$ and the product 
    \[\Pi(f, x) := \prod_{\substack{p\le x \\ p\nmid m}} \sum_{r\ge 0} \frac{h\bra{p^{r+\Ind_{p\in\PP}}}}{p^r}.\]
\end{lem}
\begin{remark}\thlabel{majorabs}
    While $\Pi(f,x)$ does not necessarily converge when $x\to\infty,$ we know its module $\abs{\Pi(f,x)}$ converges. Indeed since $\PP$ is finite and $\log \abs{1+z} = \Re(z) + O\bra{|z|^2},$ we have
    \[\log \abs{\Pi(f, x)} = \sum_{\substack{p\le x \\ p\nmid m}} \frac{\Re(h(p))}{p} + O\bra{\frac{1}{p^2}} = O(1)\] when $f$ is $\chi(n)n^{it}$-pretentious.
    
    We thus get from \thref{pretentious-majorarc}
    \[\abs{S_f\bra{x, \frac{a}{qPm}}} = x \mathfrak{S}(f) \frac{g(q)}{q} + o_q(x),\]
    where \[\mathfrak{S}(f) := \abs{\frac{G(\chi)}{1+it}}\frac{1}{Pm} \prod_{p\nmid m} \abs{\sum_{r\ge 0} \frac{h\bra{p^{r+\Ind_{p\in\PP}}}}{p^r}}\]
    is a non-zero constant that depends only on $f.$
\end{remark}

\begin{proof}
    Splitting the exponential as a character sum, we have for all $n,$
    \[\e\bra{\frac{a}{qPm} n} = \sum_{dk = qPm} \Ind_{(n,qPm) = d} \e\bra{\frac{an/d}{k}} = \sum_{dk = qPm} \Ind_{(n, qPm) = d} \frac{1}{\ph(k)} \sum_{\chi_1 \bmod k} G(\chi_1) \conj{\chi_1}(an/d),\] noting that $an/d$ is coprime with $k$ when $(n, dk) = d.$ 
    \smallbreak

    Injecting this in the definition of $S_f\bra{x, \frac{a}{qPm}},$ we have \[S_f\bra{x, \frac{a}{qPm}} = \sum_{dk = qPm} \frac{1}{\ph(k)} \sum_{\chi_1 \bmod k} G(\chi_1) \conj{\chi_1}(a) \sum_{n \le x/d} f(dn) \conj{\chi_1}(n).\] 
    By \thref{halasz-character}, when $x\to\infty$ while $q$ is fixed, the main term in this expression is given by characters $\chi_1$ derived from $\chi.$ These exist if and only if $m$ divides $k$ and $\mu^2(k/m) = (m, k/m) = 1,$ and we have in this case $G(\chi_1) = G(\chi)\mu(k/m) \chi(k/m).$ Therefore, writing $k = \ell m$ we have
    \[S_f\bra{x, \frac{a}{qPm}} = \frac{x^{1+it}}{1+it} \sum_{d\ell = qP} \frac{\mu(\ell) \chi(\ell)}{\ell m} \conj{\chi}(a) \prod_{p\le x} g_d(p) + o_q(x), \] where $g_d$ is the function $g_1$ from \thref{halasz-character} (which depends on $d$ and $k = qPm/d,$ hence the change of notation). Remembering the definition of $g_d$ and noting that $d \mapsto g_d(p)$ only depends on the $p$-adic valuation of $d$ and $k,$ we can factor the sum over $d$ as a product over primes
    \[S_f\bra{x, \frac{a}{qPm}} = \frac{x^{1+it}}{1+it} G(\chi) \frac{\conj{\chi}(a)}{m} \prod_{p\le x} \Tilde{g}(p) + o_q(x)\] where 
    $$\Tilde{g}(p) =  \begin{cases}
            \sum_{r\ge 0} \frac{h(p^r)}{p^r}, &\text{ if } p \nmid qPm, \\ 
            \frac{f(p^j)}{p^{j(1+it)}} + \chi(p^j) \sum_{r > j} \frac{h(p^{r})}{p^r} - \frac{\chi(p)}{p}\frac{f(p^{j-1})}{p^{(j-1)(1+it)}}, &\text{ if } p^j\|qP \text{ and } p\nmid m, \\
            \frac{f(p^j)}{p^{j(1+it)}}, \text{ if } p|m &\text{ and } p^j\|qP \ (j \text{ might be } 0).
    \end{cases}$$
    In the case where $p^j\|qP$ and $p\nmid m,$ we actually have $\Tilde{g}(p) = \chi(p^j) \sum_{r \ge j} \frac{h(p^{r})}{p^r}$ by definition of $h(p^j).$ The conclusion then follows from the definitions of $\Tilde{g}$ and $g.$     
\end{proof}

We also need a technical lemma to show we can ensure $S_1\bra{y,\alpha - \frac{a}{q}} \gg \frac{1}{\abs{\alpha - \frac{a}{q}}}$ is true simultaneously for many values of $q$ for the right choice of $y.$

\begin{lem}\thlabel{technical-E}
    Let $\alpha \in\R\smallsetminus\Q.$ Then for any finite set of rationals $\frac{a_1}{q_1},\ldots, \frac{a_k}{q_k},$ there exists $y\in\N$ such that $$\abs{S_1\bra{y,\alpha - \frac{a_i}{q_i}}} \ge \frac{1}{2\pi\abs{\alpha - \frac{a_i}{q_i}}}.$$
\end{lem}
\begin{proof}
    We note $\beta_i := \alpha - \frac{a_i}{q_i}.$ We have for all $i,$ $$\abs{S_1(y,\be_i)} = \abs{\sum_{n=1}^y \e\bra{\beta_i n}} = \abs{\frac{\sin (\pi\beta_i y)}{\sin (\pi\beta_i)}} \ge \frac{\abs{\sin(\pi\beta_i y)}}{\pi|\beta_i|},$$ so that the result is true as soon as all $\beta_i y$ are simultaneously at least $1/6$ away from all integers. 

    We now take $y := nQ,$ with $Q := \prod_i q_i$ and $n$ to be fixed. This way, we have for all $i,$ $$\beta_i y = \alpha Qn - a_i n \frac{Q}{q_i} \equiv \alpha Q n \mod 1.$$ It is then sufficient to take $n$ such that the fractional part of $\alpha Q n$ is between $1/6$ and $5/6,$ which exists as $\alpha Q$ is irrational.
    
\end{proof}

\subsection{Proof of \thref{setup-bazinmain}}

\begin{proof}[Proof of \thref{setup-bazinmain}]
    By \thref{new1}, we may assume $f$ is $\chi(n)n^{it}$-pretentious for some $\chi,t.$ We may also assume $\chi$ is primitive as if $\chi$ derives for $\chi',$ $f$ is also $\chi'(n)n^{it}$-pretentious. Now, assume by contradiction that \[\sum_{\substack{a,n \\ (a,nm) = 1}} \bra{\frac{|g(n)|}{n\abs{\alpha - \frac{a}{nPm}}}}^2 = +\infty.\] 
    Then for all $C > 0,$ we can find a finite set of rationals $\frac{a_1}{q_1},\ldots, \frac{a_k}{q_k}$ (where $k$ may depend on $C$) such that \[\sum_{i=1}^k \bra{\frac{|g(q_i)|}{q_i\abs{\alpha - \frac{a_i}{q_iPm}}}}^2 > C.\] 
    We then use \thref{technical-E} to get $y$ (which may depend on $C$) such that for all $i,$ \[\abs{S_1\bra{y,\alpha - \frac{a_i}{q_iPm}}} \gg \frac{1}{\abs{\alpha - \frac{a_i}{q_iPm}}},\] where the implied constant crucially does not depend on $C.$ Moreover, by \thref{majorabs} we have for all $i$ 
    \[\frac{1}{x} \abs{S_f\bra{x, \frac{a_i}{q_iPm}}} \to_{x\ra \infty} \mathfrak{S}(f)\frac{|g(q_i)|}{q_i}.\] We can thus find $x \ge \max\{y, (Pmq_1)^2,\ldots, (Pmq_k)^2\}$ (which depends on the $q_i$ and thus on $C$) such that for all $1 \le i \le k, $ \[\frac{1}{x} \abs{S_f\bra{x, \frac{a_i}{q_iPm}}} \ge \frac{\mathfrak{S}(f)}{2}\frac{|g(q_i)|}{q_i} \gg \frac{|g(q_i)|}{q_i},\] where the implied constant depends only on $f.$ 

    It then follows from \thref{large-sieve} (applied to $f\Ind_{[1,x]}$ and $Q := Pm \max_{1\le i\le k} q_i \le x^{1/2}$) that for that choice of $x$ and $y,$ we have 
    \begin{align*}
        \frac{1}{x} \sum_{-y < n < x} \abs{\sum_{\substack{n < u \le n+y \\ 1\le u \le x}} f(u) \e(\alpha u)}^2 &\gg \frac{1}{x^2} \sum_{\substack{q\le Q \\ (a,q) = 1}} \abs{S_f\bra{x,\frac{a}{q}}}^2 \abs{S_1\bra{y,\alpha - \frac{a}{q}}}^2
        \\&\gg \frac{1}{x^2} \sum_{i=1}^k \abs{S_f\bra{x, \frac{a_i}{q_iPm}}}^2 \abs{S_1\bra{y,\alpha - \frac{a_i}{q_iPm}}}^2
        \\&\gg \frac{1}{x^2} \sum_{i=1}^k \bra{x\frac{|g(q_i)|}{q_i}}^2 \frac{1}{\abs{\alpha - \frac{a_i}{q_iPm}}^2}
        \\&\gg C,
    \end{align*}
    where all the implied constants do not depend on $C.$ In other words, 
    \[\sup_{x \ge y} \frac{1}{x} \sum_{-y < n < x} \abs{\sum_{\substack{n < u \le n+y \\ 1\le u \le x}} f(u) \e(\alpha u)}^2 = +\infty.\] But if $\sup_x \abs{S_f(x,\alpha)} < \infty,$ then for $y\le x,$
        \[\frac{1}{x} \sum_{-y < n < x} \abs{\sum_{\substack{n < u \le n+y \\ 1\le u \le x}} f(u) \e(\alpha u)}^2 \le \frac{x+y}{x} \left(2\max_{n\le x} \abs{S_f(n,\al)}\right)^2 = O(1), \] contradiction.
\end{proof}

\subsection{Special cases}
We now consider the case where the exponential sum is bounded for a positive measure set of $\alpha$ (\thref{bazinmain}), as well as the case where $f$ is completely multiplicative or takes finitely many values (\thref{special}).

\begin{proof}[Proof of \thref{bazinmain}]
    If \[\sup\limits_{x}\abs{S_f(x,\al)} < \infty,\] then by \thref{setup-bazinmain} we have for some $\chi$ of conductor $m,$ $t \in\R$ and $P, g$ defined as in \thref{setup-bazinmain}, \[\sum_{\substack{a,n \\ (a,nPm) = 1}} \bra{\frac{|g(n)|}{n\abs{\alpha - \frac{a}{nPm}}}}^2 < \infty\] and in particular \[\abs{\alpha - \frac{a}{nPm}} \le \frac{|g(n)|}{nPm}\] for at most finitely many $n,a$ with $(a, nm) = 1.$ The Duffin--Schaeffer theorem \cite{duffin1941khintchine,koukoulopoulos2020duffinschaeffer} tells us this happens for a positive measure set of $\alpha$ if and only if 
    \[\sum_{n} |g(n)|\frac{\phi(nPm)}{nPm} < \infty.\] 
    As $g$ is multiplicative, this is equivalent to \[\sum_{p} \sum_{k\ge 1} |g(p^k)| \frac{p-\Ind_{p\nmid Pm}}{p} < \infty.\]
    
    We now show this is equivalent to the desired conclusion, \[\sum_{p} \sum_k \abs{f(p^k)- f(p^{k-1})\chi(p)p^{it}} = \sum_{p\nmid m} \sum_k |h(p^k)| + \sum_{p|m} \sum_k |f(p^k)| < \infty.\] 
    Recalling the definition $$g(p^r) := \chi(p^r) \frac{\sum_{k\ge 0} \frac{h\left(p^{r+k + \Ind_{p\nmid P}}\right)}{p^{k}}}{\sum_{k\ge 0} \frac{h\left(p^{k + \Ind_{p\nmid P}}\right)}{p^k}}$$ for $p\nmid m$ and $g(p^r) = f(p^r)p^{-rit}$ for $p|m,$ 
    we have
    \begin{align*}
        \sum_{p} \sum_{k\ge 1}& |g(p^k)| \frac{p-\Ind_{p\nmid Pm}}{p} \\&= \sum_{p\nmid m} \frac{1 - \frac{\Ind_{p\nmid Pm}}{p}}{\sum_{k\ge 0} \frac{h\left(p^{k + \Ind_{p\nmid P}}\right)} {p^k}} \sum_{r\ge 1 + \Ind_{p\nmid P}} \bra{|h(p^r)| + \sum_{k > r} O\bra{\frac{h(p^k)}{p^{k-r}}}} + \sum_{p|m} \sum_k |f(p^k)|
        \\&= \sum_{p\nmid m} \bra{1 + O\bra{\frac{1}{p}}} \sum_{k\ge 1 + \Ind_{p\nmid P}} |h(p^k)| \bra{1 + \sum_{r < k} O\bra{\frac{1}{p^{k-r}}}} + \sum_{p|m} \sum_k |f(p^k)|
        \\&= \sum_{p\nmid m} \sum_{k\ge 1 + \Ind_{p\nmid P}} |h(p^k)| \bra{1 + O\bra{\frac{1}{p}}} + \sum_{p|m} \sum_k |f(p^k)|. 
    \end{align*}
    Clearly, the $1 + O\bra{\frac{1}{p}}$ term and the absence of finitely many terms for $p|P, k = 1$ do not affect the convergence of the sum, hence the desired conclusion.
\end{proof}

We now consider the case where $f$ is completely multiplicative or takes finitely many values (\thref{special}), and show the exponential sum is unbounded for all irrational $\alpha$ in this case unless $f$ is exactly of the form $\chi(n) n^{it}.$

\begin{proof}[Proof of \thref{special}] 
    
    We first consider case ii). If $f: \N \ra \C$ takes finitely many values, it is not hard to show that if $f$ is $\chi(n)n^{it}$-pretentious, then in fact $t=0$ - see, for example, Lemma 5.1 of \cite{klurman2018rigidity}. 
    \smallbreak
    Since $f$ and $\chi$ take finitely many values, there is then some $\ep > 0$ such that with the notation of \thref{setup-bazinmain}, if $f(p^k) \ne f(p^{k-1})\chi(p)$, then \[\abs{f(p^k) - f(p^{k-1}) \chi(p)} > \ep.\]
    Now, we have for these $p^k,$ assuming $p\nmid Pm,$
    \[\ep \le \abs{h(p^k)} = \abs{g(p^k) -\frac{g(p^{k+1})\conj{\chi}(p)}{p}} \abs{\sum_{r} \frac{h(p^r)}{p^r}} \ll  \abs{g(p^k)} + \frac{\abs{g(p^{k+1})}}{p}.\]
    Similarly if $p\in\PP$ we have $\ep \le \abs{h(p^k)} \ll  \abs{g(p^{k-1})} + \frac{\abs{g(p^{k})}}{p},$ and if $p|m$ then $\ep \le \abs{f(p^k)} = \abs{g(p^k)}. $ Therefore, if there are infinitely many $p^k$ such that $f(p^k) \ne f(p^{k-1})\chi(p)$, there must be infinitely many $p^k$ such that $\abs{g(p^k)} > \ep'$ for some $\ep'$ (which depends only on $\ep.$)

    But for each $p$ and $k,$ we can find $a$ with $(a,p^k Pm) = 1$ and $p^k \abs{\alpha - \frac{a}{p^k Pm}} \ll_{Pm} 1,$\footnote{If $p\le 2Pm,$ take any $a$ coprime with $pPm$ and within distance $pPm \ll (Pm)^2$ of $\alpha p^k Pm.$ Otherwise, certainly there are several numbers coprime with $Pm$ within distance $Pm$ of $\alpha p^k Pm,$ and at most one of these is multiple of $p,$ so we can find a number coprime with $p^kPm$ within distance $O_{Pm}(1)$ of $\alpha p^k Pm.$} so that
    \[\sum_{\substack{a,n \\ (a,nPm) = 1}} \bra{\frac{|g(n)|}{n\abs{\alpha - \frac{a}{nPm}}}}^2 \gg_{Pm} \sum_{p} \sum_{k} \abs{g(p^k)}^2 \ge \infty\cdot {\ep^{\prime}}^2 = \infty,\] 
    hence by \thref{setup-bazinmain} $S_f(x,\alpha)$ cannot be bounded.

    We now move to case i), where $f$ is completely multiplicative. Then, assume by contradiction there is some $p_0$ such that $f(p_0) \ne \chi(p_0)p_0^{it}$ and $f(p_0) \ne 0.$ With the notations of \thref{setup-bazinmain}, we then have for $k\ge 1,$ if $p_0\nmid m,$
    \[h(p_0^k)\chi(p_0^k)p_0^{kit} = f(p_0^k) - f(p_0^{k-1}) \chi(p_0) p_0^{it} = f(p_0)^{k} \bra{1 - \frac{\chi(p_0) p_0^{it}}{f(p_0)}}\] so that $p\not\in\PP$ and for $r\ge 1,$
    \begin{align*}
        g(p_0^r) &= \frac{f(p_0)^r\bra{1 - \frac{\chi(p_0) p_0^{it}}{f(p_0)}} \sum_{k\ge 0} \bra{\frac{f(p_0) \conj{\chi}(p_0) p_0^{-it}}{p_0}}^k }{1+ \bra{1 - \frac{\chi(p_0) p_0^{it}}{f(p_0)} } \sum_{k\ge 1} \bra{\frac{f(p_0) \conj{\chi}(p_0) p_0^{-it}}{p_0}}^k} 
        \\&= \frac{f(p_0)^r \bra{1 - \frac{\chi(p_0) p_0^{it}}{f(p_0)}} }{1 - \frac{f(p_0) \conj{\chi}(p_0) p_0^{-it}}{p_0} + \bra{1 - \frac{\chi(p_0) p_0^{it}}{f(p_0)}}\frac{f(p_0) \conj{\chi}(p_0) p_0^{-it}}{p_0}}
        \\&= f(p_0)^r \bra{1 - \frac{\chi(p_0) p_0^{it}}{f(p_0)}} \frac{p_0}{p_0 - 1}.
    \end{align*} 
    Therefore, since for all $r$ we can find $a$ with $(a,p_0^r m) = 1$ and $p_0^r \abs{\alpha - \frac{a}{p_0^r m}} \ll_{m} 1,$ we have as $|f(p_0)|=1,$
    \[\sum_{\substack{a,n \\ (a,nm) = 1}} \bra{\frac{|g(n)|}{n\abs{\alpha - \frac{a}{nm}}}}^2 \gg_{m} \sum_{r} \abs{g(p_0^r)}^2 = \infty, \] hence by \thref{setup-bazinmain} $S_f(x,\al)$ cannot be bounded. If $p_0|m$ we have $\abs{g(p_0^r)} = \abs{f(p_0)}^r = 1$ and obtain the same conclusion.

    Thus, for all $p_0$ such that $f(p_0) \ne \chi(p_0)p_0^{it},$ we have $f(p_0) = 0.$ If there are infinitely many such $p_0,$ they all verify $g(p_0) = \chi(p_0) \frac{p_0}{p_0 - 1},$ so that we again have 
    \[\sum_{\substack{a,n \\ (a,nm) = 1}} \bra{\frac{|g(n)|}{n\abs{\alpha - \frac{a}{nm}}}}^2 \gg \sum_{p_0 : p_0\nmid m, f(p_0) = 0} \abs{g(p_0)}^2 = \infty. \]    
    It thus remains the case where there are finitely many such $p_0,$ in which case we actually have $f(n) = \chi'(n) n^{it}$ for some non-primitive character $\chi'$ derived from $\chi.$
\end{proof}

\section{Sharpness of main results}\label{5}

In this section we prove \thref{sharp,sharper}, which show \thref{stronger,bazinmain} are close to optimal in some sense. In particular, \thref{special} is not true for general multiplicative functions, as any multiplicative function verifying $f(p^k) = 1 + O(p^{-2k})$ furnishes a counterexample by \thref{sharp}.

\begin{prop}\thlabel{sharp}
    Define the "degree of pretentiousness" of $f$ as 
    \[\sigma(f) := \inf_{\chi,t} \sigma_a\bra{f * \mu\chi n^{it}} \in [-\infty, 1],\]
    where $h_{f,g} := f * \mu g$ and $\sigma_a(h)$ denotes the abscissa of absolute convergence of $h.$

    Then we have the following:
    \begin{enumerate}[label=\roman*)]
        \item If $\sigma(f) > 0,$ then \[\sup\limits_{x}\abs{S_f(x,\al)} = \infty\] for all $\alpha.$
        \item If $\sigma(f) < 0,$ then \[\sup\limits_{x}\abs{S_f(x,\al)} < \infty\] for almost all $\alpha.$
    \end{enumerate}
\end{prop}

\begin{prop}\thlabel{sharper}
    Assume there is a Dirichlet character $\chi$ and $t\in\R$ such that with $h := f * \mu\chi n^{it},$ we have \[\sum_d \abs{h(d)} \log \bra{2 + \frac{1}{|h(d)|}} < \infty.\] Then
    \[\sup\limits_{x}\abs{S_f(x,\al)} < \infty\] for almost all $\alpha.$
\end{prop}
\begin{proof}
    First, we note that for all irrational $\alpha$ and for all $x,$ we have with $h = f * \mu\chi n^{it}$ and by \thref{character2}
    \[S_f(x,\al) = \sum_{d\le x} h(d) \sum_{n \le x/d} \chi(n) n^{it} \e(\alpha d n) = O_{m,t}\bra{\sum_{d\le x} \frac{|h(d)|}{|\e(\alpha dm) - 1|}},\] where $m$ is the conductor of $\chi.$ Therefore it is sufficient to show $$\sum_{d} \frac{|h(d)|}{|\e(\alpha dm) - 1|} < \infty$$ for almost all $\al.$
    
    For $D \in \N,$ define $$S(D) := \{\alpha \in [0,1] : \exists d \ge D, \abs{\e(\alpha dm) - 1} < |h(d)|.\}$$
    Clearly $S(D)$ is of measure $$\lambda(S(D)) \le \sum_{d\ge D} 2\abs{h(d)} \, \ra_{D\ra\infty} 0,$$ as $\sum_{d} \abs{h(d)} < \infty.$
    Moreover, for $d \ge D,$ we have $$\int_{[0,1] \smallsetminus S(D)} \frac{1}{\abs{\e(\al dm) - 1}} \ll \log \bra{\frac{1}{|h(d)|}},$$ so that 
    \[\int_{[0,1] \smallsetminus S(D)} \sum_{d\ge D}  \frac{|h(d)|}{|\e(\alpha dm) - 1|} \ll \sum_{d\ge D} \abs{h(d)} \log \bra{\frac{1}{|h(d)|}} < \infty.\]
    In particular, $$\sum_{d} \frac{|h(d)|}{|\e(\alpha dm) - 1|} < \infty$$ for almost all $\alpha\not\in S(D).$ As the measure of $S(D)$ tends to zero as $D$ tends to infinity, it follows the above is actually true for almost all $\alpha,$ as desired.
\end{proof}

Using \thref{stronger} and \thref{sharper} we may now prove \thref{sharp}, which summarizes our results.

\begin{proof}[Proof of \thref{sharp}]
    We prove i) by contraposition. If the exponential sum is unbounded for some $\alpha,$ then by \thref{stronger} there is a Dirichlet character $\chi$ and a real number $t \in \R$ such that
    \[\sum_{p} \max_{k}\abs{f(p^k)-\chi(p^k)p^{ikt}} < \infty.\]
    In particular with $h := f * \mu\chi n^{it},$ we have $$\abs{h(p^k)} = \abs{f(p^k) - f(p^{k-1}) \chi(p)p^{it}} \le \abs{f(p^k)-\chi(p^k)p^{ikt}} + \abs{f(p^{k-1})-\chi(p^{k-1})p^{i(k-1)t}},$$ so that for all $\ep > 0,$
    \[\sum_{p} \sum_k \frac{\abs{h(p^k)}}{p^{\ep k}} \le \sum_p 2 \max_{k}\abs{f(p^k)-\chi(p^k)p^{ikt}} \cdot \sum_k \frac{1}{p^{\ep k}} < \infty.\] Thus, $h$ has an abscissa of absolute convergence $\le 0,$ so that $\sigma(f) \le 0.$

    We now prove ii). If $\sigma(f) < 0,$ then by definition we can find $\chi$ and $t$ such that $h$ defined as above has abscissa of absolute convergence $< 0.$ In particular, $\sum_{d} \abs{h(d)} \log d < \infty,$ so that
    \[\sum_d \abs{h(d)} \log \bra{2 + \frac{1}{|h(d)|}} \ll \sum_{d: |h(d)| > 1/d^2} \abs{h(d)} \log d + \sum_{d: |h(d)| \le 1/d^2} \frac{\log d}{d^2} < \infty.\] It thus follows from \thref{sharper} that $S_f(x,\al)$ is almost always bounded in this case.
\end{proof}

\section*{Acknowledgements}
The second and third authors would like to thank Oleksiy Klurman for putting them in touch with each other, and for many useful discussions and conversations, as well as suggesting the problem. The first author was added to this paper after the third author visited Paris at the invitation of R{\'e}gis de la Bret{\`e}che, which led to the authors proving a stronger version of \thref{stronger,special} and a new result (\thref{bazinmain}). We would like to thank R{\'e}gis de la Bret{\`e}che for his hospitality and encouragement during the visit, as well as helpful discussions.  F.T. would additionally like to thank Christopher Atherfold and Besfort Shala for interesting discussions, Misha Rudnev for moral support, and C{\'e}dric Pilatte for discussion on \thref{conjecture}. The authors also thank  R{\'e}gis de la Bret{\`e}che, Oleksiy Klurman and Besfort Shala for useful comments on an earlier version of this paper.
\smallbreak
The third author is supported by a Heilbronn Doctoral Partnership.

\printbibliography
\end{document}